\documentclass[11pt]{article}

\usepackage{geometry}
\usepackage{amsmath,latexsym,amstext,amsfonts}
\usepackage[english]{babel}
\usepackage{graphicx}
\usepackage{caption}
\usepackage{makeidx}
\usepackage{pstricks,pst-grad,pst-plot,pst-node}
\usepackage{pstricks-add}
\usepackage{mathdots}
\usepackage{braket}
\usepackage{tikz}
\usetikzlibrary{snakes}
\usetikzlibrary{arrows}

\newcommand{\N}{\mathbb{N}}
\newcommand{\I}{\mathbb{I}}
\newcommand{\F}{\mathbb{F}}

\newcommand{\Drel}{\mathbin{\mathcal D}}

\def \RR{{\mathcal R}}
\def \HH{{\mathcal H}}
\def \LL{{\mathcal L}}
\def \DD{{\mathcal D}}
\def \JJ{{\mathcal J}}

\newtheorem{theorem}{Theorem}[section]

\newtheorem{corollary}[theorem]{Corollary}

\newtheorem{lemma}[theorem]{Lemma}

\newtheorem{proposition}[theorem]{proposition}
\newtheorem{obs}[theorem]{Remark}

\newtheorem{proof}[theorem]{Proof}

\title{On Ideals of a Skew Lattice}
\author{\textit{Jo\~ao Pita Costa}\bigskip\\
University of Ljubljana,\\
Faculty of Mathematics and Physics,\\
Jadranska 19, 1000 Ljubljana, Slovenia.\\
joaopitacosta@gmail.com
\footnote{The author thanks the support
of Funda\c c\~ao para a Ci\'encia e Tecnologia with the reference SFRH / BD / 36694
/ 2007.}
}
\date{\today}

\begin{document}

\maketitle

\begin{abstract}
Ideals are one of the main topics of interest to the study of the order structure of an algebra. Due to their nice properties, ideals have an important role both in lattice theory and semigroup theory.
Two natural concepts of ideal can be derived, respectively, from the two concepts of order that arise in the context of skew lattices. 
The correspondence between the ideals of a skew lattice, derived from the preorder, and the ideals of its respective lattice image is clear. Though, skew ideals, derived from the partial order, seem to be closer to the specific nature of skew lattices. 
In this paper we review ideals in skew lattices and discuss the intersection of this with the study of the coset structure of a skew lattice.  
\end{abstract}

\section{Introduction}

As skew lattices are a generalization of lattices, the order structure has an important role in the study of these algebras. 
Skew lattices can be seen as double regular bands where two different order concepts can be defined:  the natural preorder, denoted by $\preceq $, and the natural partial order, denoted by $\leq $, one weaker then the other and both of them motivated by analogous order concepts defined for bands. They generalize the partial order of the correspondent lattice. Both of these ideals are frequent in the literature.  
Though, unlike lattices, the admissible Hasse diagram representing the order structure of a skew lattice can not determine its algebraic structure. The study of the coset structure of skew lattices began with Leech in \cite{Le93} mentioned as \textit{(global) coset geometry}. It derives from Leech's first decomposition theorem and gives a perspective into the role of the partitions that $\DD$-classes determine on each other providing important additional information. Several varieties of skew lattices were characterized using laws involving only cosets (cf. \cite{Co09a}, \cite{Co11} and \cite{JPC12}). 

Two kinds of ideals can be naturally derived from these preorders. The strong concept of ideal is naturally derived from the preorder and has been largely studied having an important role in the research centered on the congruences of skew Boolean algebras with intersections, a particular case of the Boolean version of a skew lattice (cf. \cite{Le96} and \cite{BL}). The skew ideals, on the other hand, tend to describe the \textit{skew nature} of skew lattices, having a clear relation with the natural preorder and with the concept of normality. 
They were not so well explored in the literature and are still a \textit{fresh} subject of research. 
Our motivation is due to the important role of ideals in the study of skew Boolean algebras. They reveal their usefulness on the recent generalization of the \textit{Stone duality} for skew lattices in \cite{Ba11} where the authors denoted these two versions of ideals as $\leq$-\textit{ideals} and $\preceq$-\textit{ideals}. In this paper we will name them \textit{skew ideals} and \textit{ideals}, respectively.
In this work we study both kinds of ideals in skew lattices and the intersections of this theory with the coset structure of a skew lattice. 

The reader that is not familiar with the theory of skew lattices can find in \cite{Le96}, \cite{Le93} and \cite{Le90} a relevant review on the subject.
For further reading on the order structure of semigroups and lattices, two important motivations to the results in the following discussion, the author suggests \cite{Ho76}, \cite{Ba40} and \cite{Gr71}.

\section{Preliminaries}

A \emph{skew lattice} $\mathbf{S}$ is a set S equipped with two associative binary operations $\vee$  and $\wedge$, called \emph{meet} and \emph{join}, respectively. Skew lattices satisfy the absorption laws $$(y\wedge x)\vee x = x = x\vee (x\wedge y)$$ and their duals. Both $\wedge$ and $\vee$ are idempotent. If we would admit all the possible absorption laws we would get the commutativity of the operations and, thus, lattices. Indeed, the commutativity of one operation implies the commutativity of the other.
Given that $\vee$ and $\wedge$ are associative, idempotent operations, skew lattices are characterized by the following absorption dualities: $x\vee y=x$ iff $x\wedge y = y$ and $x\vee y = y$ iff $x\wedge y=x$ (cf. \cite{Le89}). Green's relations are five equivalence relations characterizing the elements of a semigroup in terms of the principal ideals they generate.
If $\mathbf{S}$ is a band, the operation $\DD $ equals $\JJ$ and $ \HH$ is the trivial equivalence.  
Due to this, the Green's relations can be simplified for bands as : $x\RR y \text{  iff  } xy=y \text{ and } yx=x$; $x\LL y \text{  iff  } xy=x \text{ and } yx=y$; $x\DD y \text{  iff  } xyx=x \text{ and } yxy=y$. 
Moreover, due to the absorption dualities, the Green's relations in the context of skew lattices are defined in \cite{Le89} by: $\RR=\RR_{\wedge}=\LL_{\vee}$;  $\LL=\LL _{\wedge}=\RR_{\vee}$;  $\DD =\DD _{\wedge}=\DD_{\vee}$. 
\emph{Right-handed} skew lattices are the skew lattices for which $\RR =\DD$ while left-handed skew lattices are determined by $\LL =\DD$. Thus, right-handed skew lattices are the ones satisfying the identity $x\wedge y\wedge x = y\wedge x$ or, equivalently, $x\vee y\vee x = x\vee y$. Left-handed skew lattices can be defined by similar identities: $x\wedge y\wedge x = x\wedge y$ or, equivalently, $x\vee y\vee x = y\vee x$.

Given nonempty sets $L$ and $R$ the direct product $L\times R$ together with the operations $(x,y)\vee (x',y') = (x',y)$ and $(x,y)\wedge (x',y') = (x,y')$ forms a skew lattice. A \emph{rectangular skew lattice} is an  isomorphic copy of this skew lattice. Leech's First Decomposition Theorem states that the Green's relation $\DD$ is a congruence on any skew lattice $\mathbf S$, while $\DD$-equivalence classes are exactly the maximal rectangular subalgebras of $\mathbf S$ and $\mathbf S/\DD$ is the maximal lattice image of $\mathbf S$.
Moreover, a skew lattice $\mathbf S$ may be viewed as the conjugation of its pair of regular band reducts $(S,\wedge )$ and $(S,\vee )$ whose binary operations dualize each other as described, that is, $x\wedge y= y\vee x$ holds for every $x,y\in S$. Leech's Second Decomposition Theorem states that both reducts $(S;\wedge)$ and $(S;\vee)$ are regular bands and that every skew lattice factors as the fibred product of a right-handed skew lattice with a left-handed skew lattice over their common maximal lattice image $\mathbf{S}/\DD $, with both factors being unique up to isomorphism (cf. \cite{Le89}).

A \emph{primitive} skew lattice is a skew lattice composed by two comparable $\DD $-classes $A$ and $B$ such that $A>B$, denoted by $\set{A>B}$. A \emph{skew diamond} is a skew lattice composed by two incomparable $\DD $-classes, $A$ and $B$, a join class $J=A\vee B$ and a meet class $M=A\wedge B$. It is usually denoted by $\set{J>A,B>M}$. Moreover, $\set{J>A}$, $\set{J>B}$, $\set{A>M}$ and $\set{B>M}$ are primitive skew lattices. Recall that a \emph{chain} (or \emph{totally ordered set}) is a set where each two elements are (order) related, and an \emph{antichain} is a set where no two elements are (order) related. We call $\mathbf{S}$ a \emph{skew chain} whenever $\mathbf S/\DD $ is a chain. 
Any sub lattice $\mathbf{T}$ of $\mathbf{S}$ intersects each $\DD $-class of $\mathbf{S}$ in at most one point. If $\mathbf{T}$ meets each $D$-class of $\mathbf{S}$ in exactly one point, then $\mathbf{S}$ is called a \emph{lattice section} of $\mathbf{S}$. As such, it is a maximal sub lattice that is also an internal copy inside $\mathbf{S}$ of the maximal lattice image $\mathbf{S}/\DD $ (cf. \cite{Le93}). Each lattice section of $\mathbf{S}$ is isomorphic to $\mathbf{S}/\DD $. 

A skew lattice is said to be \emph{symmetric} if it is biconditionally commutative, that is if it satisfies: $x\wedge y= y\wedge x \text{   iff   } x\vee y= y\vee x.$
Most of the ``well behaved'' skew lattices considered are symmetric. All skew chains are symmetric (cf. \cite{Le92}). Another relevant property, extensively studied in skew lattices is \textit{normality}. 
A skew lattice $\mathbf{S}$ is said to be \emph{normal} if it satisfies $x\wedge y\wedge z\wedge w= x\wedge z\wedge y\wedge w$. Dually, skew lattices that satisfy $x\vee y\vee z\vee w= x\vee z\vee y\vee w$ are named \emph{conormal}. 
Skew lattices that are simultaneously normal and conormal, are called \emph{binormal}. 
Schein described in \cite{Sc83} binormal skew lattices as algebras that factor as the product of a lattice with a rectangular skew lattice.
When $\mathbf{S}/\DD$ is a distributive lattice we say that $\mathbf{S}$ is a \emph{quasi-distributive} skew lattice. 
If $\mathbf{S/\DD}$ is finite and quasi-distributive, then $\mathbf{S}$ has a lattice section which is naturally isomorphic to $\mathbf{S/\DD}$ (cf. \cite{Le86}). A skew lattice is \emph{distributive} if it satisfies 
$x\wedge(y\vee z)\wedge x= (x\wedge y\wedge x)\vee (x\wedge z\wedge x)$ and 
$x\vee (y\wedge z)\vee x= (x\vee y\vee x)\wedge (x\vee z\vee x)$.
All distributive skew lattices are quasi-distributive. Moreover, in the presence of normality, distributivity is equivalent to quasi-distributivity (cf. \cite{Le96}).

A \emph{skew Boolean algebra} is an algebra $\mathbf{S }= (S ; \vee, \wedge, \backslash, 0)$ of type $<2, 2, 2, 0>$ such that $(S ; \vee, \wedge, 0)$ is a distributive, normal and symmetric skew lattice with $0$, and $\backslash $ is a binary operation on $\mathbf{S}$ satisfying $(x \wedge y \wedge x) \vee (x\backslash y) = x$ and $(x \wedge y \wedge x) \wedge (x\backslash y) = 0$. Skew Boolean algebras form a variety of skew lattices (cf. \cite{Le90}).

\section{Order Structure}\label{Order Structure}

Influenced by the natural preorders defined for bands of semigroups in \cite{Ho76}, we define in a skew lattice $\mathbf{S}$ the following distinct generalizations for the lattice order:
\begin{itemize}
\item[] the \emph{natural partial order} by $x\geq y$ if $x\wedge y=y=y\wedge x$ or dually, $x\vee y = x = y\vee x$;
\item[] the \emph{natural preorder} by $x\succeq y$ if $y\wedge x\wedge y = y$ or, dually, $x\vee y\vee x = x$.
\end{itemize} 

Observe that $x\Drel y$ iff $x\succeq y$ and $y\succeq x$. Furthermore, the following technical result will be useful for understanding the theorems ahead.

\begin{proposition}\cite{Le89}\label{list_order}
Let $\mathbf{S}$ be a skew lattice and $x,y,z\in S$. Then,
\begin{itemize}
\item[(i)] $x\wedge y\preceq x,y$ and $x,y\preceq x\vee y$;
\item[(ii)] $x\wedge y\wedge x\leq x\leq x\vee y\vee x$;
\item[(iii)] $x\wedge y\leq y\vee x$;
\end{itemize}
\end{proposition}

It is well known the relation between the \textit{algebraic structure} and the \textit{order structure} of a lattice. Analogously to what occurs in lattice theory, the partial order defined in a skew lattice has an important role in the study of these algebras (cf. \cite{Jo61}, \cite{Co80} or \cite{Le90}). Though, the natural partial order doesn't always determine a skew lattice, as we shall see. 
The following characterization of the natural partial order for skew lattices is an easy but useful observation.

\begin{lemma}\cite{AAA80}\label{order_id} 
Let $\mathbf{S}$ be a skew lattice and $x,y\in S$. Then $x\geq y$ iff $y=x\wedge y\wedge x$ or, dually, $x=y\vee x\vee y$. 
\end{lemma}

Usually $\DD $ is referred in the available literature as the \emph{natural equivalence} although we know it to be a congruence in any skew lattice. A $\DD$-class containing an element $x$ shall be denoted by $\DD_x$. 

\begin{proposition}\cite{Le89}\label{preorder}
Let $\mathbf{S}$ be a skew lattice with $\DD $-classes $A>B$. 
For all $x,y\in S$, $x\leq y$ implies $x\preceq y$. 
Furthermore, whenever $x\in A$, $y\in B$, $$B\leq A\text{  iff  }x\preceq y.$$ 
\end{proposition}

\begin{proposition}\cite{Le89}\label{up}   
Let $A$ and $B$ be comparable $\DD$-classes in a skew lattice $\mathbf S$ such that $A \geq B$. Then, for each $a \in A$, there exists $b \in B$ such that $a \geq b$, and dually, for each $b \in B$, there exists $a \in A$ such that $a \geq b$. 
\end{proposition}

\begin{proposition}\cite{Le89}\label{comutdiam}     
Let $\set{J>A,B>M}$ be a skew diamond. 
Then, for every $a\in A$ there exists $b\in B$ such that $a\vee b=b\vee a$ in $J$ and $a\wedge b=b\wedge a$ in $M$.
Moreover, $$J=\{a\vee b \mid a\in A,  b\in B \text{ and }  a\vee b=b\vee a\}\text{ and } M=\set{a\wedge b \mid a\in A, b\in B \text{ and } a\wedge b=b\wedge a}.$$ 
\end{proposition}

\begin{proposition}\cite{Le89}\label{order_dclass}
Let $\mathbf{S}$ be a skew lattice and $x,y\in S$. Then $x\geq y$ together with $x\DD y$ implies $y=x$. In other words, all $\DD $-classes are antichains.
\end{proposition}

The latter result clarifies the relation between the rectangular structure and the order structure of a skew lattice. On the other hand, proposition \ref{preorder} expresses that the natural preorder is \emph{admissible}, in the sense of Wechler (cf. \cite{We92}), with respect to the natural partial order. 
The fact that $\DD$ can be expressed by the natural preorder $\preceq$ allows us to draw diagrams that are capable of representing skew lattices, named \emph{admissible Hasse diagrams}, where the natural partial order is indicated by full edges and the congruence $\DD $ is indicated by dashed edges. 
Skew lattice operations are not uniquely defined by the natural partial order, unlike what happens in the lattice case. Hence, the admissible Hasse diagram expresses partial information of the skew lattice. The admissible Hasse diagram in Figure \ref{fig_nc5L} represents both the right-handed skew lattice determined by the following \emph{Cayley tables} that define the operations in each sided case: 
\medskip

\begin{center}

$\begin{array}{lcl}

\text{Right-handed skew lattice $\mathbf{NC}_{5}^R$:}\medskip

&

&

\\

\begin{tabular}{ l | c c c c c}
  $\wedge$ & 0 & $a$ & $b$ & $c$ & 1 \\
  \hline
  $0$                & 0 & 0 & 0 & 0 & 0 \\
  $a$                & 0 & $a$ & $b$ & 0 & $a$ \\
  $b$                & 0 & $a$ & $b$ & $0$ & $b$ \\
  $c$                & 0 & 0 & $0$ & $c$ & $c$ \\
  $1$                & 0 & $a$ & $b$ & $c$ & 1  
\end{tabular}

&

&

\begin{tabular}{ l | c c c c c}
  $\vee$      & 0 & $a$ & $b$ & $c$ & 1 \\
  \hline
  0                    & 0 & $a$ & $b$ & $c$ & 1 \\
  $a$                & $a$ & $a$ & $a$ & 1 & 1 \\
  $b$                & $b$ & $b$ & $b$ & 1 & 1 \\
  $c$                & $c$ & 1 & 1 & $c$ & 1 \\
  1                    & 1 & 1 & 1 & 1 & 1  
\end{tabular}

\vspace{0.8cm}

\end{array}$

$\begin{array}{lcl}

\text{Left-handed skew lattice $\mathbf{NC}_{5}^L$:}\medskip

&

&

\\

\begin{tabular}{ l | c c c c c}
  $\wedge$ & 0 & $a$ & $b$ & $c$ & 1 \\
  \hline
  0                & 0 & 0 & 0 & 0 & 0 \\
  $a$                & 0 & $a$ & $a$ & 0 & $a$ \\
  $b$                & 0 & $b$ & $b$ & 0 & $b$ \\
  $c$                & 0 & 0 & 0 & $c$ & $c$ \\
  1                & 0 & $a$ & $b$ & $c$ & 1  
\end{tabular}

&

&

\begin{tabular}{ l | c c c c c}
  $\vee$      & 0 & $a$ & $b$ & $c$ & 1 \\
  \hline
  0                & 0 & $a$ & $b$ & $c$ & 1 \\
 $a$                & $a$ & $a$ & $b$ & 1 & 1 \\
  $b$                & $b$ & $a$ & $b$ & 1 & 1 \\
  $c$                & $c$ & 1 & 1 & $c$ & 1 \\
  1                & 1 & 1 & 1 & 1 & 1  
\end{tabular}

\end{array}$
\end{center}

\medskip

The Cayley tables completely determine the skew lattice but the diagram itself could represent two different skew lattices: one left-handed skew lattice and other right-handed skew lattice.

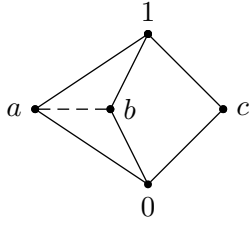
\begin{figure}[here]
\begin{center}

$\begin{array}{lcr}

&

\begin{pspicture}(-2,-2)(2,2)

\psline[linewidth=0.5 pt,linestyle=dashed]{*-*}(-0.5,0)(-1.5,0)

\psline[linewidth=0.5pt]{*-*}(-1.5,0)(0,1)
\psline[linewidth=0.5 pt]{*-*}(-0.5,0)(0,1)
\psline[linewidth=0.5pt]{*-*}(1,0)(0,1)

\psline[linewidth=0.5pt]{*-*}(1,0)(0,-1)
\psline[linewidth=0.5 pt](-0.5,0)(0,-1)
\psline[linewidth=0.5 pt](-1.5,0)(0,-1)

\uput[0](-0.5,0){$b$}
\uput[180](-1.5,0){$a$}
\uput[1](1,0){$c$}
\uput[90](0,1){$1$}
\uput[270](0,-1){$0$}

\end{pspicture}

&

\end{array}$

\caption{\small \sl The admissible Hasse diagram of both the skew lattices $\mathbf{NC}_{5}^R$ and $\mathbf{NC}_{5}^L$.} \label{fig_nc5L} 

\end{center}  
\end{figure}

\section{Classical Ideals and Skew Ideals}\label{Classical Ideals and Skew Ideals}

A \emph{preordered set} $\mathbf{S}=(S,\preceq)$ (or $\mathbf{S}=S_{\preceq}$) is a nonempty set $S$ endowed with a preorder $\preceq$. When the preorder is a partial order $\leq$ we call $\mathbf S$ a \emph{partially ordered set} (or \emph{poset}) and write $\mathbf{S}=(S,\leq)$ (or $\mathbf{S}=S_{\leq}$), instead. All skew lattices are preordered by $\preceq $ and partially ordered by $\leq$.
A nonempty subset $I$ [$F$] of a preodered set $\mathbf P$ is \emph{downwards [upwards] -directed} if, whenever $x\in S$, $y\in I$ [$z\in F$] and $x\leq y$ [$x\leq x$], then $x\in I$ [$x\in F$]. 
A \emph{poset lattice ideal} $I$ [\emph{filter} $F$] of a poset $\mathbf P$ is any nonempty downwards [upwards] -directed subset of $P$ .   
When $I$ [$F$] is also closed under the join operation, $\vee $ [meet operation, $\wedge$], then it is called an \emph{lattice ideal} [\emph{filter}].
Filters are often regarded as \textit{dual ideals} (cf. \cite{Gr71}).

In the following paragraphs we present the classical ideals of a skew lattice, derived from the natural preorder $\preceq$. 

A nonempty subset $I$ of a skew lattice $\mathbf{S}$ closed under $\vee $ is an \emph{ideal} of $\mathbf{S}$ if, for all $x\in S$ and $y\in I$, $x\preceq y$ implies $x\in I$.
Regarding that $\DD $ is a congruence in any skew lattice $\mathbf S$, whenever $I$ is a sub skew lattice of $\mathbf S$, we will use the following notation: $$ I/\DD =\set{\DD_{x}:x\in I}.$$
As lattices are skew lattices that coincide with its lattice image, we always look at a lattice as the lattice image of a skew lattice. Thus, whenever $I$ is a lattice ideal, $\bigcup I/\DD = I.$

\begin{proposition}\label{charac_weak}
Let $\mathbf{S}$ be a skew lattice and $I$ a subset of $\mathbf{S}$ closed under $\vee$. The following statements equivalently define an ideal:
\begin{itemize}
\item[$i)$] for all $x\in S$ and $y\in I$, $x\preceq y$ implies $x\in I$ ;
\item[$ii)$] for all $x\in S$ and $y\in I$, $y\wedge x , x\wedge y\in I$;
\item[$iii)$] for all $x\in S$ and $y\in I$, $x\wedge y\wedge x\in I$.
\end{itemize}
\end{proposition}

\begin{proof}
Let us show that $(i)\Rightarrow (ii)\Rightarrow (iii)\Rightarrow (i)$.
Let $x\in S$ and $y\in I$. 
If $(i)$ holds, propositionosition \ref{list_order} implies that $y\wedge x , x\wedge y\preceq y$ so that $y\wedge x , x\wedge y\in I$.
Now, assuming $(ii)$ we get $x\wedge y\wedge x=(x\wedge y)\wedge (y\wedge x)\in I$. 
Finally, from the assumptions in $(ii)$ follows that $x\wedge y\wedge x\in I$ and $x\wedge y\wedge x\preceq y$. Thus, $x=x\wedge y\wedge x\in I$, completing the proof.
\end{proof}

Analogously, a nonempty subset $F$ closed under $\vee$ of a skew lattice $\mathbf{S}$ is a \emph{filter}  of $\mathbf{S}$ if one of the following equivalent statements holds:
\begin{itemize}
\item[$i)$] for all $x\in F$ and $y\in S$, $x\preceq y$ implies $y\in F$ ;
\item[$ii)$] for all $x\in F$ and $y\in S$, $y\vee x , x\vee y\in F$;
\item[$iii)$] for all $x\in F$ and $y\in S$, $ y\vee x\vee y\in F$.
\end{itemize}  

The notion of ideals of skew lattices was first mentioned in \cite{BL} where $(ii)$ is used, and defined in \cite{Le08} using $(iii)$ in the context of skew Boolean algebras. 
Later was defined  in \cite{Ba11} using $(i)$ and mentioned as a subset that is \textit{lower with respect to} $\preceq$. 
In this last paper the theory of ideals of skew lattices revealed to be a relevant tool to reach the Stone duality variation for skew Boolean algebras with intersections (that is, skew Boolean algebras with finite greatest lower bound with respect to the partial order $\leq$) for which this theory is fairly ``well-behaved'' (cf. \cite{Ba11}).  
Motivated by the definition of lattice ideals in \cite{Gr71} we present an alternative characterization of ideals and filters.

\begin{proposition}\label{weak_Gr}
Let $\mathbf{S}$ be a skew lattice and $I$ a nonempty subset of $S$. Then $I$ is an ideal iff the following equivalence holds for all $a,b\in S$

$$ a,b\in I \Leftrightarrow a\vee b\vee a\in I.$$

Analogously, a nonempty subset $F$ of $S$ is a filter iff for all $a,b\in S$, $$ a,b\in F \Leftrightarrow a\wedge b\wedge a\in F.$$
\end{proposition}

\begin{proof}
Let us suppose that $I$ is an ideal of $\mathbf{S}$. The direct implication is obvious. Let $a,b\in S$ such that $a\vee b\vee a\in I$. As $a,b\preceq a\vee b\vee a$ then $a,b\in I$.

Conversely suppose that the equivalence holds. Thus, $I$ is closed under $\vee$. Let $a\in I$ and $b\in S$ such that $b\preceq  a$. Then $a\vee b\vee a=a\in I$ and therefore $b\in I$.

The proof regarding filters is similar. 
\end{proof}

\begin{corollary}\label{weak_sub}
All ideals and filters in a skew lattice are sub skew lattices.
\end{corollary} 

\begin{proof}
Let $I$ be an ideal of a skew lattice $\mathbf{S}$.  $I$ is a subset of $S$ closed under the operation $\vee $, by definition.  
On the other hand, if $x,y\in I$ then $x\wedge y\preceq x$ implies $x\wedge y\in I$.
The proof regarding filters is similar.
\end{proof}

\begin{proposition}\label{id_full}
Let $\mathbf{S}$ be a skew lattice, $I$ an ideal of $\mathbf{S}$ and $F$ a filter of $\mathbf{S}$. If $x\in S$, $y\in I$ and $x\DD y$, then $x\in I$. Similarly, if $x\in S$, $y\in F$ and $x\DD y$, then $x\in F$. 
\end{proposition}

\begin{proof}
If $I$ is an ideal and $x\in S$, whenever $y\in I$ is such that $x\DD y$, then $x\preceq y$ implying that $x\in I$.
The proof for filters is analogous.  
\end{proof}

\begin{corollary}\label{id_sd}
Let $\mathbf{S}$ be a skew lattice and $I, F\subseteq S$ being unions of $\DD$-classes of $\mathbf S$. 
Then, $I/\DD$ is a lattice ideal of $\mathbf{S}/\DD$ iff $I$ is an ideal of $\mathbf{S}$; dually, $F/\DD$ is a lattice filter of $ \mathbf{S}/\DD$ iff $F$ is a filter of $\mathbf{S}$. 
\end{corollary}

\begin{proof}
This result is a direct consequence of propositions \ref{weak_Gr} and \ref{id_full}.
\end{proof}

Let us denote by $\I (S)$ the set of all ideals and by $\F (S)$ the set of all filters of a skew lattice $\mathbf S$. The result bellow was proved in \cite{Sl73a} for Slav\' ik's version of skew lattices and also holds in Leech's version studied in this dissertation.

\begin{theorem}\cite{BL}\label{slavik}
Let $\mathbf S$ be a skew lattice. Then, $(\I (S); \cap, \cup)$ and $(\F (S); \cap, \cup)$ are complete lattices (with respect to the set inclusion) isomorphic to the lattices of all ideals and the lattice of all filters of the lattice $\mathbf S/\DD $, respectively. 
\end{theorem}

A nonempty subset $I$ of $\mathbf{S}$ closed under $\vee $ is a \emph{skew ideal} of $\mathbf{S}$ if, for all $x\in S$ and $y\in I$,  $x\leq y$ implies $x \in I$. 
Dually, a nonempty subset $F$ of $\mathbf{S}$ closed under $\wedge$ is a \emph{skew filter} of $\mathbf{S}$ if, for all $x\in S$ and $y\in F$,  $x\geq y$ implies $x \in F$.

\begin{obs}
Due to propositionosition \ref{preorder}, all ideals are examples of skew ideals.
The set $\set{0,a,c}$ in Figure \ref{fig_nc5L} is also an example of a skew ideal that is not an ideal, illustrating that not all skew ideals are ideals.
\end{obs}

\begin{proposition}\label{id-sub}
All skew ideals [filters] are sub skew lattices.
\end{proposition}

\begin{proof}
Let $\mathbf{S}$ be a skew lattice and $I$ a skew ideal of $\mathbf S$.
By definition, $I$ is closed under $\vee $. We shall see that $I$ is also closed under $\wedge$.
Let $x,y\in I$. As $x\wedge y\wedge x\leq x$ and $y\wedge x\wedge y\leq y$, both $x\wedge y\wedge x$ and $y\wedge x\wedge y$ are in $I$.
But $x\wedge y=(x\wedge y\wedge x)\wedge (y\wedge x\wedge y)=(y\wedge x\wedge y)\vee (x\wedge y\wedge x)$ where the second equality follows from the fact that  $(x\wedge y\wedge x)\DD (y\wedge x\wedge y)$.
Hence, $x\wedge y\in I$ as required. The proof for $y\wedge x\in I$ is similar.
The case of skew filters is analogous.
\end{proof}

Based on the order characterization for the partial order described in proposition \ref{order_id} we characterize skew ideals and filters of a skew lattice as follows:

\begin{proposition}\label{sideal_charac}
A nonempty subset $I$ of $\mathbf S$ is a skew ideal of $\mathbf S$ iff 
\begin{itemize}
\item[$(i)$] for all $x,y\in I$, $x\vee y\in I$; 
\item[$(ii)$] for all $x\in S$ and $y\in I$,  $y\wedge x\wedge y\in I$. 
\end{itemize}
\end{proposition}

\begin{proof}
Let $x\in S$ and $y\in I$. As $y\wedge x\wedge y\leq y$ and $I$ is a skew ideal containing $y$ it follows that $y\wedge x\wedge y\in I$.
Conversely, let $x\in S$ and $y\in I$ be such that $y\wedge x\wedge y\in I$ and $x\leq y$. Then $x=y\wedge x\wedge y \in I$ as required.
\end{proof}

An analogous proof verifies the following characterization of skew filters.

\begin{proposition}\label{sfilter_charac}
A nonempty subset $F$ of $\mathbf S$ is a skew filter of $\mathbf S$ iff 
\begin{itemize}
\item[$(i)$] for all $x,y\in I$, $x\wedge y\in F$; 
\item[$(ii)$] for all $x\in S$ and $y\in F$,  $y\vee x\vee y\in F$. 
\end{itemize}
\end{proposition}

Denote by $\I^{*} (S)$ [$\F^{*} (S)$] the set of all skew ideals [filters] of a skew lattice $\mathbf S$.

\begin{proposition}
Let $\mathbf S$ be a skew lattice. Then, $(\I^{*}(S), \wedge ,\vee )$ and $(\F^{*}(S), \wedge ,\vee )$ are complete lattices considering the $\wedge $ operation to be the usual intersection of sets. 
\end{proposition}

\begin{proof}
$\mathbf S$ is an ideal and, therefore, a skew ideal. Let us show that the intersection of skew ideals is a skew ideal.
Let $\set{I_{k}:k\in K}$ be a class of skew ideals of a skew lattice $\mathbf S$.
Clearly $\bigcap_{k\in K}I_{k}$ is closed under $\vee $. 
Let $x\in S$ and $y\in \bigcap_{k\in K}I_{k}$ such that $x\leq y$. 
The fact that $y$ is simultaneously a member of all skew ideals $I_{k}$ ensures that $x\in \bigcap_{k\in K}I_{k}$. 
The result then follows. 
The proof regarding filters is analogous.
\end{proof}

\begin{obs}
In the case of skew ideals one can easily see that, in general, the skew ideal does not cover the $\DD $-classes it intersects, that is, $x\Drel y$ with $x\in S$ and $y\in I$ does not necessarily imply that $x\in I$. To see that just consider the left-handed skew lattice $NC_5^L$ given by the Cayley tables in Figure \ref{fig_nc5L}. 
For instance, in that example, $a\Drel b$ and $a\notin I=\set{0,b}$. 
Moreover, $NC_5^L$ has $9$ skew ideals while its lattice image has only $4$ ideals.
\end{obs}

\section{Principal Ideals and Normality}\label{Principal Ideals and Normality}

Let $X$ be a nonempty subset of a lattice $\mathbf{S}$. Let $Y$ be the set of all [skew] ideals of $\mathbf{S}$ containing $X$. The intersection $M$ of all elements in $Y$ is also a [skew] ideal of $\mathbf{S}$ that contains $X$ is called the [skew] \emph{ideal generated by} $X$, denoted by $X^{\downarrow} $ [$X^{\downarrow ^{*}}$]. If $X$ is a singleton and $x\in X$, then $M$ is said to be a \emph{principal} [skew] \emph{ideal} generated by $x$, written $x^{\downarrow} $ [$x^{\downarrow ^{*}}$]. Principal [skew] filters have an analogous definition and are denoted by $x^{\uparrow} $ [$x^{\uparrow ^{*}}$]. 

\begin{proposition}\label{sxs}
Let $\mathbf{S}$ be a skew lattice and $x\in S$. Then,

\begin{itemize}
\item[$i)$] $x^{\downarrow}  =S\wedge x\wedge S=\set{y\in S:y\preceq x}$ and $x^{\uparrow} =S\vee x\vee S=\set{y\in S:x\preceq y}$. 
\item[$ii)$]  $x^{\downarrow ^{*}} = x\wedge S\wedge x=\set{y\in S:y\leq x}$ and $x^{\uparrow ^{*}}=x\vee S\vee x=\set{y\in S:x\leq y}$. 
\end{itemize}

\end{proposition}

\begin{proof}
$(i)$ Let us first show that $S\wedge x\wedge S=\set{y\in S:y\preceq x}$.
Fix $a\in S$. Thus, $a\wedge x\wedge b\preceq x$ due to regularity: $a\wedge x\wedge b\wedge x\wedge a\wedge x\wedge b=a\wedge x\wedge b\wedge a\wedge x\wedge b=a\wedge x\wedge b$.
Conversely,  if $a\preceq x$ then $a=a\wedge x\wedge a\in S\wedge x\wedge S$. 
The equality $S\vee x\vee S=\set{y\in S:x\preceq y}$ has an analogous proof. 

Now we will show that $x^{\downarrow}  =S\wedge x\wedge S$.
Let $y,z\in S\wedge x\wedge S$, that is, $y,z\preceq x$. 
Then $x\vee y\vee x=x$ and $x\vee z\vee x=x$ so that $x=x\vee x=x\vee y\vee x\vee x\vee z\vee x=x\vee y\vee x\vee z\vee x=x\vee y\vee z\vee x$ due to regularity. 
Similarly, $x=x\vee z\vee y\vee x$. Hence, $y\vee z,z\vee y\preceq x$ and therefore $y\vee z,z\vee y\in S\wedge x\wedge S$.
Let $y\in S$ and $z\in S\wedge x\wedge S$ such that $y\preceq z$. 
Fix $a,b\in S$ such that $z=a\wedge x\wedge b$. 
Thus, $y\preceq a\wedge x\wedge b$ so that $y=y\wedge a\wedge x\wedge b\wedge y\in S\wedge x\wedge S$. 
By idempotency, $x=x\wedge x\wedge x\in S\wedge x\wedge S$.
Let $I$ be an ideal of $\mathbf S$ such that $x\in I$. 
Let $a,b\in S$. Then, $a\wedge x\wedge b\preceq x\in I$ so that $a\wedge x\wedge b\in I$ and, therefore, $S\wedge x\wedge S\subseteq I$.
The proof regarding principal filters is analogous. 

$(ii)$ Let us first show that $x\wedge S\wedge x=\set{y\in S:y\leq x}$.
Fix $a\in S$. Thus, $x\wedge a\wedge x\leq x$.
Conversely,  let $a\in S$ such that $a\leq x$. Then $a=x\wedge a\wedge x\in x\wedge S\wedge x$.  
The proof of $x\vee S\vee x=\set{y\in S:x\leq y}$ is analogous. 

Now we will show that $x^{\downarrow ^{*}} = x\wedge S\wedge x$.
Let $y,z\in x\wedge S\wedge x$. Let $a,b\in S$ such that $y=x\wedge a\wedge x$ and $z=x\wedge b\wedge x$. By absorption, $y\vee z\vee x=(x\wedge a\wedge x)\vee (x\wedge b\wedge x)\vee x= (x\wedge a\wedge x)\vee x=x$. Similarly $x\vee y\vee z= x$. Hence $y\vee z\leq x$ and its an analogous proof to show that $z\vee y\leq x$. Therefore $y\vee z,z\vee y\in x\wedge S\wedge x$.
Now let $a,b\in S$ such that $a\leq x\wedge b\wedge x$. Then $a=x\wedge b\wedge x\wedge a\wedge x\wedge b\wedge x\in x\wedge S\wedge x$.  
By idempotency, $x=x\wedge x\wedge x\in x\wedge S\wedge x$.
Let $I$ be an ideal of $\mathbf S$ such that $x\in I$. Let $a\in S$. As $x\wedge a\wedge x\leq x\in I$ then $x\wedge a\wedge x\in I$ so that $x\wedge A\wedge x\subseteq I$. 
The proof regarding principal skew filters is analogous. 
\end{proof}

\begin{corollary}
Let $\mathbf{S}$ be a skew lattice and $x\in S$. Then,
$x\wedge S\wedge x \subseteq S\wedge x\wedge S$ and, dually, $x\vee S\vee x \subseteq S\vee x\vee S$.
\end{corollary}

The following result show us how principal ideals and principal skew ideals of a skew lattice $\mathbf S$ relate with the decomposition in $\DD$-classes of $S$.

\begin{proposition}\label{antecip}
Let $\mathbf{S}$ be a skew lattice $X$ a nonempty subset of $S$ and $x, y\in S$. Then,

\begin{itemize}
\item[$i)$] $ X^{\downarrow} =\bigcup \set{\DD_x^{\downarrow}: x\in X} $ and $ X^{\uparrow} =\bigcup \set{\DD_x^{\uparrow}: x\in X} $.
\item[$ii)$] $x^{\downarrow} =y^{\downarrow} \Leftrightarrow \DD_x=\DD_y $ and $x^{\uparrow} =y^{\uparrow} \Leftrightarrow \DD_x=\DD_y $.
\end{itemize}

\end{proposition}

\begin{proof}
$(i)$ This is a direct consequence of proposition \ref{id_full}: as the ideals in $\mathbf{S}$ are just the unions of the $\DD$-classes that constitute the elements of the ideals of $S/\DD$, then the principal ideals are just unions of blocks constituting the correspondent principal lattice ideal.

$(ii)$ is due to the fact that when $x$ and $y$ generate the same ideal, $x\preceq y$ and $y\preceq x$, that is, $x\in \DD_y $. The converse is a direct consequence of proposition \ref{id_full} as $x\DD y$ implies $x^{\downarrow} \subseteq y^{\downarrow}$. 
\end{proof}

Such a characterization for skew ideals analogous to the one in proposition \ref{antecip} $(i)$ is a much more difficult aim regarding the nature of such algebras.
The recent work on a generalization of Stone duality for skew Boolean algebras with intersections brought up a characterization of skew ideals generated by a nonempty set $X$ recurring to the closure of finite joins of the downward closure of $X$:

\begin{proposition}\cite{Ba11}\label{skewdi}
Let $\mathbf{S}$ be a $\wedge$-distributive skew lattice and $\emptyset \neq X\subseteq S$. Then, $$X^{\downarrow ^{*}}=\set{x_{1}\vee \dots \vee x_{n}\,:\, \text{ for all } i\leq n\, \exists y_{i} \in X\text{ such that } x_{i}\leq y_{i}}.$$
\end{proposition}

\begin{corollary}
Let $\mathbf{S}$ be a $\wedge$-distributive skew lattice and $x\in S$. Then,
$$x^{\downarrow ^{*}}=\set{x_{1}\vee \dots \vee x_{n}\,:\, \text{ for all } i\leq n\, \text{ and } x_{i}\leq x}.$$
\end{corollary}

\begin{proposition}\label{antecip2}
Let $\mathbf{S}$ be a skew lattice $X$ a nonempty subset of $S$ and $a, b\in S$. Then,

\begin{itemize}
\item[$i)$] $x^{\downarrow ^{*}} \cap \DD_x=\set{x}$ and $x^{\uparrow ^{*}} \cap \DD_x=\set{x}$.  
\item[$ii)$] $x^{\downarrow ^{*}} =y^{\downarrow ^{*}} \Leftrightarrow x=y$ and $x^{\uparrow ^{*}} =y^{\uparrow ^{*}} \Leftrightarrow x=y$.
\end{itemize}

\end{proposition}

\begin{proof}  
$(i)$ Let $a\in x^{\downarrow ^{*}} \cap \DD_x$. Then, $a\leq x$ and $a\Drel x$ so that $a=x$.

$(ii)$ Suppose that $x^{\downarrow ^{*}}=y^{\downarrow ^{*}}$. Then, $x\leq y$, $y\leq x$ from which we get $x=y$. 
\end{proof}

\begin{proposition}\label{disjointunion}
Let $\mathbf S$ be a skew lattice. 
For all $x\in S$, the principal ideal $ x^{\downarrow }$ is the union of all principal skew ideals $y^{\downarrow ^{*}}$ such that $y\in \DD_{x}$.
\end{proposition}

\begin{proof}
Let $x\in S$.
Due to proposition \ref{antecip} $(ii)$, $x^{\downarrow }=\bigcup \DD_{x}^{\downarrow}$.  
Whenever $a\in x^{\downarrow }$ and $a\notin \DD_{x}$, then $\DD_{a}\leq \DD_{x}$. proposition \ref{up} implies the existence of $y\in \DD_{x}$ such that $y\geq a$. Thus, $a\in y^{\downarrow ^{*}}$.
On the other hand, if $a\in y^{\downarrow ^{*}}$ with $y\DD x$, then $a\leq y\preceq x$ so that $a\in x^{\downarrow}$.
\end{proof}

\begin{proposition}\label{disjointunion2}
Let $\mathbf{S}$ be a skew lattice and $x, y\in S$ such that $x\Drel y$. Then, $|x^{\downarrow ^*}|=|y^{\downarrow ^*}|$.
\end{proposition}

\begin{proof}
Consider the maps $\phi: x\downarrow ^*\rightarrow y\downarrow ^*$ and $\varphi: y\downarrow ^*\rightarrow x\downarrow ^*$ defined by $\phi(a)= y\wedge a\wedge y$ and $\varphi(b)=x\wedge b\wedge x$, for every $a\in x\downarrow ^*$ and $b\in y\downarrow ^*$.
Both of these maps are clearly well defined. Let us see that they are the inverse of each other.
Let $c\in y\downarrow ^*$. Then, $$\phi \circ  \varphi(c)= y\wedge x\wedge c\wedge x\wedge y=y\wedge x\wedge y\wedge c\wedge y\wedge x\wedge y=y\wedge c\wedge y=c$$ 
due to regularity and the assumption that $x\DD y$.

Thus, $\phi \circ \varphi = \Delta_{y\downarrow ^*}$ and, similarly, $\varphi \circ \phi = \Delta_{x\downarrow ^*}$.
\end{proof}

\begin{corollary}\label{index_ideal}
Let $\mathbf{S}$ be a skew lattice and $x\in S$.
Then, $$|x^{\downarrow} |\leq |\DD_{x}|.|x^{\downarrow ^*}|.$$ 
\end{corollary}

The following results show us how principal ideals and principal skew ideals are familiar algebras and strengthen the relation between the study of ideals and the study of normality.

\begin{proposition}\cite{Le92}\label{normalid}
A skew lattice $\mathbf{S}$ is normal iff each sub skew lattice $x^{\downarrow ^*}$ is a sub lattice of $\mathbf{S}$. 
Dually, $\mathbf{S}$ is conormal iff each sub skew lattice $x^{\uparrow ^*}$ is a sub lattice of $\mathbf{S}$.
\end{proposition}

\begin{corollary}\cite{Le92}\label{normz}
Let $\mathbf{S}$ be a normal skew lattice. Then, $\mathbf{S}$ is quasi-distributive if, and only if, for all $x\in S$, $x^{\downarrow ^{*}}$ is distributive.
\end{corollary}
 
Normal skew lattices were further studied in \cite{Le90} and \cite{Le92} and, due to proposition \ref{normalid}, are sometimes cited as \textit{local lattices} \cite{Le93} or as \textit{mid commutative skew lattices} \cite{Le92}. 
The \emph{center} of $\mathbf{S}$ is the sub skew lattice $Z(S)$ formed by the union of all its singleton $D$-classes. $Z(S)$ is always a normal sub skew lattice of  $\mathbf S$. Moreover, the center $Z(\mathbf{S})$ of a normal skew lattice $\mathbf S$, if nonempty, corresponds to an ideal in $\mathbf{S}/\DD$. 
In particular, whenever the lattice image $\mathbf{S}/\DD$ of $\mathbf{S}$ has minimal elements, $Z(\mathbf{S})$ is empty precisely when those minimal classes of $\mathbf{S}$ are nontrivial.s \cite{Le92}.

\begin{proposition}\cite{Le90}    
If a normal skew lattice $\mathbf{S}$ has top $\DD$-class with maximal element $m$, then $S=m^{\downarrow }$. Furthermore, the sub lattice $m^{\downarrow ^{*}}$ is a lattice section of the underlying lattice $\mathbf{S}/\DD$ in $\mathbf{S}$.
\end{proposition}

In the beginning of the nineties, Leech used principal skew ideals when defining a skew Boolean algebra in his review paper \cite{Le96} as follows:  a skew Boolean algebra is any symmetric skew lattice with zero $(S;\wedge, \vee, 0)$ such that, for all $x\in S$, $x^{\downarrow ^{*}}$ is a Boolean algebra.

\section{Ideals and Cosets}

Consider a skew lattice $\mathbf{S}$ consisting of exactly two $\DD$-classes $A>B$. Given $b\in B$, the subset $A\wedge b\wedge A=\{a\wedge b\wedge a \,|\, a\in A \}$ of $B$ is said to be a \emph{coset} of $A$ in $B$ (or \emph{$A$-coset in $B$}). Similarly, a coset of $B$ in $A$ (or $B$-coset in $A$) is the subset $B\vee a\vee B=\{b\vee a\vee b \,|\, b\in B \}$ of $A$, for a fixed $a\in A$.  Given $a\in A$, the \emph{image set} of $a$ in $B$ is the set $$a\wedge B\wedge a = \set{a \wedge b\wedge a\,|\,b\in B}=\set{b\in B\,|\,b< a}.$$ Dually, given $b\in B$ the set $b\vee A\vee b = \set{a\in A:b<a}$ is the image set of $b$ in $A$.

Cosets are irrelevant in both the context of semigroup theory or lattice theory being something very specific to skew lattices. In fact, the coset structure reveals a new perspective that does not have a counterpart either in the theory of lattices or in the theory of bands of semigroups (cf. \cite{JPC12}). The following theorem gives us a further perspective on this decomposition.

\begin{theorem}\cite{Le93} \label{coset_part}
Let $\mathbf{S}$ be a skew lattice with comparable $\DD$-classes $A>B$. Then, $B$ is partitioned by the cosets of $A$ in $B$ and, dually, $A$ is partitioned by the cosets of $B$ in $A$. 
The image set of any element $a\in A$ in $B$ [$b\in B$ in $A$] is a transversal of the cosets of $A$ in $B$ [$B$ in $A$]. 
Furthermore, any coset $B\vee a\vee B$ of $B$ in $A$ is isomorphic to any coset $A\wedge b\wedge A$ of $A$ in $B$ under a natural bijection $\varphi $ defined implicitly by: 

\begin{center}
$x\in B\vee a\vee B \mapsto y\in A\wedge b\wedge A$ if and only if $x\geq y$. 
\end{center}

The operations $\wedge$ and $\vee$ in $A\cup B$ are determined jointly by these coset bijections and the rectangular structure of each $\DD$-class. 
\end{theorem}

As a consequence of Theorem \ref{coset_part} that describes the image set of any element $b\in B$ as a transversal of cosets of $B$ in $A$, for any $b,b'\in B$, $$b^{\downarrow ^{*}}= b'^{\downarrow ^{*}}.$$
Dually, given any $a,a'\in A$, their images in $B$ have equal powers. 
This allows us to define the \emph{index} of $B$ in $A$, denoted by $\left[ A:B \right]$, as the cardinality of the image set $b\vee A\vee b$, for any $b\in B$. Dually, we define the index of $A$ in $B$, denoted by $[B:A]$, as the cardinality of the image set $a\wedge B\wedge a$, for any $a\in A$. The index $\left[ A:B \right]$ equals the cardinality of the set of all $B$-cosets in $A$, and $[B:A]$ equals the cardinality of the set  of all $A$-cosets in $B$.

\begin{proposition}\cite{Le93}\label{cs_normal} 
Let $\mathbf{S}$ be a skew lattice. Then, $\mathbf{S}$ is normal iff for each comparable pair of $\DD $-classes $A>B$ in $\mathbf{S}$,  $B$ is the entire coset of $A$ in $B$. That is, for all $x,x'\in B$, 

\begin{center}
$A\wedge x\wedge A=A\wedge x'\wedge A$.
\end{center}

Dually, $\mathbf{S}$ is conormal iff for all comparable pair of $\DD $-classes $A>B$ in $\mathbf{S}$ and all $x,x'\in A$, $B\vee x\vee B=B\vee x'\vee B$.
\end{proposition}

A skew lattice is \emph{categorical} if nonempty composites of coset bijections are coset bijections. 
Skew Boolean algebras are examples of categorical skew lattices \cite{Le93}.   
 
\begin{proposition}\cite{AAA80}\label{order_eq_cosets}
Let $\mathbf{S}$ be a categorical skew lattice and $A\geq B$ comparable $\DD $-classes. Then,

$$\bigcup \set{\phi_{a,b}: a\in A,b\in B} =\geq _{A\times B},$$ where each $\phi _{a,b}$ is the coset bijection from $B\vee a\vee B$ to $A\wedge b\wedge A$.

\end{proposition}

A categorical skew lattice is \emph{strictly categorical} if the compositions of coset bijections between comparable $\DD $-classes $A>B>C$ are never empty. Normal skew lattices are strictly categorical \cite{Le11a}.

\begin{proposition}\cite{Le11a}\label{lstrictly}
Let $\mathbf{S}$ be a skew lattice. Then, $\mathbf{S}$ is strictly categorical if, and only if, given any $x\in S$, $x^{\uparrow ^{*}}$ is a normal sub skew lattice of $\mathbf{S}$ and $x^{\downarrow ^*}$ is a conormal sub skew lattice of $\mathbf{S}$.
\end{proposition}

In particular, sub skew lattices of strictly categorical skew lattices are also strictly categorical. Furthermore,

\begin{theorem}\cite{Co11}\label{count_cat}
Given any skew chain $A>B>C$ in a skew lattice $\mathbf{S}$, $[C:A]\leq [C:B][B:A]$. If $\mathbf{S}$ is a strictly categorical skew lattice and both $A$ and $C$ are finite, then so is $B$ and

$$[C:A] = [C:B][B:A]$$ \label{CAT} 

In general, consider any skew chain $A_{1}<A_{2}<\dots <A_{n}$,  for some $n\in \N$, in a strictly categorical skew lattice $\mathbf{S}$. If $A_{1}$ and $A_{n}$ are finite, then so are all intermediate $\DD $-classes and

$$[A_1:A_n] =\pi_{k=1}^{n-1} [A_k:A_{k+1}].$$
\end{theorem}

Due to Theorem \ref{count_cat}, when we consider $\DD$-classes $A_{1}>A_{2}>\dots >A_{n}=C$ of a strictly categorical skew lattice $\mathbf S$ and $e_{i}$ be the number $[C,A_{i}]$, we get $e_{k+1}\mid e_{k}$ for all $k\in \set{1,\dots, n}$. 

\red{(...)} \black

\smallskip

Let $\mathbf{S}$ be a skew lattice and $x,y\in S$. We denote $x\lhd y$ whenever $x<y$ and there is no $w\in S$ such that $x<w<y$.
The following results show how skew ideals are related with the coset structure of a skew lattice.

\begin{proposition}\label{idealtocoset}
Let $\mathbf{S}$ be a skew lattice and $x,y\in S$ such that $y\leq x$. 
Then, $x^{\downarrow ^*}$ intersects all cosets of $\DD_{x}$ in $\DD_{y}$.
Dually, $y^{\uparrow ^*}$ intersects all cosets of $\DD_{y}$ in $\DD_{x}$.
\end{proposition}

\begin{proof}
This is a consequence of the existence of the family of coset bijections between $\DD_{y}\vee x\vee \DD_{y}$ and $\DD_{x}\wedge z\wedge \DD_{x}$, guaranteed by Theorem \ref{coset_part}, for all $z\in \DD_{y}$.
The dual case is similar. 
\end{proof}

\begin{proposition}\label{ideal_coset}
Let $\mathbf{S}$ be a primitive skew lattice with comparable $\DD$-classes $A>B$. For all $a\in A$ and $b\in B$,
\begin{center}
$a^{\downarrow ^{*}} =(a\wedge B\wedge a)\cup \set{a}$ and $b^{\uparrow ^{*}}= (b\vee A\vee b)\cup \set{b}$.
\end{center}
\end{proposition}

\begin{proof}
This result is a direct consequence of proposition \ref{up} considering the image set of $a\in A$ as a transversal of cosets of $B$ in the context of $S=A\cup B$.
\end{proof}

The description of a principal skew ideal of a finite skew lattice comes as a direct consequence of Theorem \ref{coset_part} together with proposition \ref{ideal_coset}.
Observe that the case of primitive skew lattices is described in proposition \ref{ideal_coset} by the existence of $b\in B$ such that $b\lhd a$ for all $a\in A$.
The general case is described as follows: 

\begin{proposition}\label{sideal_{rec}} 
Let $\mathbf{S}$ be a finite skew lattice and $x\in S$.  Then, $x^{\downarrow ^{*}}=\bigcup _{n} x^{\downarrow ^{*}}_{n}$ where $x^{\downarrow ^{*}}_{n}$ is defined recursively as follows:

\begin{itemize}
\item[]$x^{\downarrow ^{*}}_{0}=\set{x}$;
\item[]$x^{\downarrow ^{*}}_{n+1}=\cup \set{z\wedge \DD_{y}\wedge z: z\in x^{\downarrow ^{*}}_{n}\text{ and } y\lhd z}$.
\end{itemize}

A similar result can be stated for skew filters.
\end{proposition}

\begin{proof}
Let $a\in x^{\downarrow ^*}$. If $a=x$ then proposition \ref{antecip2} $(i)$ implies that $a\in x^{\downarrow ^*}_{0}$.
If not, then $a\in S$ such that $a< x$. 
In the case that $a\lhd x$, Theorem \ref{coset_part} and proposition \ref{ideal_coset} imply that $a\in x\wedge D_{a}\wedge x\subseteq x^{\downarrow ^{*}}_{0}$.
In the case that $z$ exists such that $a\lhd z\lhd x$, considering $\DD_{a}\cup \DD_{z}$ we get $a\in (z\wedge \DD_a \wedge z)$. As $z\in x^{\downarrow ^*}_{1}$ then $(z\wedge \DD_a \wedge z) \subseteq z^{\downarrow ^*}\subseteq x^{\downarrow ^*}_{n}$; and so on.
Due to the finitude of $\mathbf S$ these iterations must finnish at some point. 
Hence, $a\in x^{\downarrow ^{*}}_{m}$, for some $m\in \N$.

Conversely, proposition \ref{antecip} $(i)$ implies that $x^{\downarrow ^{*}}_{0}\subseteq x^{\downarrow ^{*}}$. Fix $z\in x^{\downarrow ^*}_{n}$ and let $a,y\in S$ such that $y\lhd z$ and $a\in z\wedge \DD_{y}\wedge z$, that is, $a\leq z$. As $z\leq x$ then $a\leq x$ so that $y\in x^{\downarrow ^*}$.  
\end{proof}

\begin{corollary}\label{index_ideal2}
Let $\mathbf{S}$ be a finite skew lattice and $x\in S$.
If $x^{\downarrow ^*} $ is finite, then $$\left| x^{\downarrow ^*} \right | \leq 1 + [\DD_{x_{0}}:\DD_{x_{1}}] + [\DD_{x_{0}}:\DD_{x_{1}}][\DD_{x_{1}}:\DD_{x_{2}}] + \dots + [\DD_{x_{0}}:\DD_{x_{1}}]\dots [\DD_{x_{n-1}}:\DD_{x_{n}}],$$
for $x=x_{0}\lhd x_{1}\lhd \dots \lhd x_{n-1}\lhd x_{n}$ in $\mathbf{S}$.
Moreover, if $\mathbf{S}$ is strictly categorical, then $$\left| x^{\downarrow ^* }\right| \leq 1+\sum_{k=0}^{n}[\DD_{x_{0}}:\DD_{x_{k}}].$$
\end{corollary}

\begin{proof}
A direct consequence of propositions \ref{sideal_{rec}} and \ref{count_cat}.
\end{proof}

\section*{Agknowledgements}
The author would like to thank to K. Cvetko-Vah for her guidance and clarification, to Ganna Kudryavtseva for her motivation and involvement in this research, to J. Leech for the ideas and suggestions, and to M.J. Gouveia for the opportune discussions on the topic.
Finally, the author thanks the support of Funda\c c\~ao para a Ci\^ encia e Tecnologia with the reference SFRH/BD/36694/2007.

\bibliographystyle{plain}
\bibliography{Bib_20120620}

\end{document}